\documentclass[a4paper,11pt]{amsart}
\usepackage{amsmath,amsthm,amsfonts,amssymb,mathrsfs}
\usepackage{enumerate}
\usepackage[bookmarks=false]{hyperref}
\usepackage{geometry}
\usepackage[utf8]{inputenc}
\usepackage[demo]{graphicx}
\usepackage[up]{caption}
\usepackage{subcaption}
\usepackage{pgf,tikz}
\usepackage[toc,page]{appendix}
\usetikzlibrary{arrows}
\usetikzlibrary{patterns}
\usepackage{color}

\theoremstyle{plain}
\newtheorem{letterthm}{Theorem}[section]

\newtheorem{lettercor}[letterthm]{Corollary}
\newtheorem{thm}{Theorem}[section]

\newtheorem{lem}[thm]{Lemma}
\newtheorem{prop}[thm]{Proposition}

\theoremstyle{definition}
\newtheorem{defn}[thm]{Definition}
\newtheorem*{defn*}{Definition}

\newtheorem*{question*}{Question}

\newtheorem{rem}[thm]{Remark}

\numberwithin{equation}{section}

\newcommand{\DG}{\Delta\Gamma}
\newcommand{\pG}{\partial\Gamma}
\newcommand{\N}{\mathbb{N}}
\newcommand{\R}{\mathbb{R}}
\newcommand{\C}{\mathbb{C}}
\newcommand{\Z}{\mathbb{Z}}

\newcommand{\cB}{\mathcal{B}}

\newcommand{\cN}{\mathcal{N}}

\newcommand{\cR}{\mathcal{R}}

\newcommand{\cU}{\mathcal{U}}

\newcommand{\Stab}{\operatorname{Stab}}

\newcommand{\conv}{\operatorname{conv}}

\newcommand{\Ad}{\operatorname{Ad}}

\newcommand{\BS}{\operatorname{BS}}

\newcommand{\SL}{\operatorname{SL}}

\newcommand{\Prob}{\operatorname{Prob}}
\newcommand{\HNN}{\operatorname{HNN}}
\newcommand{\rk}{\operatorname{rk}}

\newcommand{\defin}{singular }

\begin{document}

\title[Maximal amenable subalgebras from maximal amenable subgroups]{Maximal amenable von Neumann subalgebras arising from maximal amenable subgroups}

\author{R\'emi Boutonnet \& Alessandro Carderi}
\email{rboutonnet@ucsd.edu}
\email{alessandro.carderi@ens-lyon.fr}
\thanks{Research partially supported by ANR grant Neumann. R.B. is also partially supported by NSF Career Grant DMS 1253402.}

\begin{abstract}
We provide a general criterion to deduce maximal amenability of von Neumann subalgebras $L\Lambda \subset L\Gamma$ arising from amenable subgroups $\Lambda$ of discrete countable groups $\Gamma$. The criterion is expressed in terms of $\Lambda$-invariant measures on some compact $\Gamma$-space. The strategy of proof is different from S. Popa's approach to maximal amenability via central sequences \cite{Po83}, and relies on elementary computations in a crossed-product C$^*$-algebra.
\end{abstract}

\maketitle

\section*{Introduction}
A (separable) finite von Neumann algebra $A \subset B(H)$ is said to be {\it amenable} if there exists a state $\varphi$ on $B(H)$, which is $A$-central, meaning that $\varphi(xT) = \varphi(Tx)$ for all $x \in A$ and all $T \in B(H)$. Moreover, this definition does not depend on the choice of the Hilbert space $H$ on which $A$ is represented.

Amenability has always played a central role in the study of von Neumann algebras. First it is a source of isomorphism, via the fundamental result of A. Connes \cite{Co76} that amenable implies hyperfinite, and the uniqueness of the hyperfinite II$_1$-factor. Amenability is also at the core of the concepts of solidity and strong solidity defined in \cite{Oz04,OP10a}. It is hence very natural to try to understand the maximal amenable subalgebras of a given finite von Neumann algebra.

In this direction, R. Kadison asked in the 1960's the following question: is any maximal amenable subalgebra of a II$_1$-factor necessarily a factor? S. Popa solved this problem in \cite{Po83}, producing an example of a maximal amenable subalgebra of the free group factor $LF_n$ which is abelian. The subalgebra in question is generated by one of the free generators of the free group $F_n$. This striking result led to more questions, refining Kadison's question: what if the ambient II$_1$-factor is McDuff? has property (T)? More generally can one provide concrete examples of maximal amenable subalgebras in a given II$_1$-factor? Some progress on this topic have been made recently.

By considering infinite tensor products of free group factors, J. Shen constructed in \cite{Sh06} an abelian, maximal amenable subalgebra in a McDuff II$_1$-factor. In \cite{CFRW10}, it is proved that the subalgebra of the free group factor generated by the symmetric laplacian operator (the {\it radial subalgebra}) is maximal amenable. In \cite{Ho14}, C. Houdayer provided uncountably many non-isomorphic examples of abelian maximal amenable subalgebras in II$_1$-factors. Last year the authors showed in \cite{BC14} that any infinite maximal amenable subgroup in a hyperbolic group $\Gamma$ gives rise to a maximal amenable von Neumann subalgebra of $L\Gamma$. Note that hyperbolic groups can have property (T). This last result is a rigidity type result in the spirit of Popa's approach to study von Neumann algebras in terms of their ``building data''.

\begin{question*}
Assume that $\Lambda < \Gamma$ is a maximal amenable subgroup. Under which conditions is $L\Lambda$ maximal amenable inside $L\Gamma$?
\end{question*}

In this paper, we provide a general sufficient condition ensuring this rigidity phenomenon.

\begin{defn*}
Consider an amenable subgroup $\Lambda$ of a discrete countable group $\Gamma$. Suppose that $\Gamma$ acts continuously on the compact space $X$. We say that $\Lambda$ is \textit{\defin} in $\Gamma$ (with respect to $X$) if for any $\mu \in \Prob_\Lambda(X)$ and $g \notin \Lambda$, we have $g \cdot \mu \perp \mu$. 
\end{defn*}

In the above definition, $\Prob_\Lambda(X)$ denotes the set of $\Lambda$-invariant Borel probability measures on $X$.

\begin{letterthm}
\label{criterion}
Let $\Gamma$ be a countable discrete group and $\Lambda<\Gamma$ be an amenable singular subgroup.
Then for any trace preserving action $\Gamma \curvearrowright (Q,\tau)$ on a finite amenable von Neumann algebra, $Q \rtimes \Lambda$ is maximal amenable inside $Q \rtimes \Gamma$.
\end{letterthm}

The conclusion of the above theorem implies in particular that
\begin{itemize}
\item $L\Lambda$ is maximal amenable inside $L\Gamma$ (case where $Q = \C$);
\item for any free measure preserving action on a probability space $\Gamma \curvearrowright (Y,\nu)$, the orbit equivalence relation $\cR(\Lambda \curvearrowright (Y,\nu))$ is maximal hyperfinite inside $\cR(\Gamma \curvearrowright (Y,\nu))$ (case where $Q = L^\infty(Y,\nu)$).
\end{itemize}

\begin{lettercor}
In the following examples, $\Lambda$ is singular inside $\Gamma$, so that the conclusion of Theorem \ref{criterion} holds.
\begin{enumerate}[(1)]
\item $\Gamma$ is a hyperbolic group and $\Lambda$ is an infinite maximal amenable subgroup \cite{BC14};
\item $\Lambda$ is any amenable group with an infinite index subgroup $\Lambda_0$, and $\Gamma = \Lambda \ast_{\Lambda_0} \Lambda'$, for some other group $\Lambda'$ containing properly $\Lambda_0$;
\item $\Gamma = \SL_n(\Z)$ and $\Lambda$ is the subgroup of upper triangular matrices.
\end{enumerate}
\end{lettercor}

Point \textit{(2)} above was proved independently by B. Leary \cite{Le14} for more general von Neumann algebras (not only group algebras).

Regarding the question of providing abelian, maximal amenable subalgebras in a given von Neumann algebra, we can prove the following. The result is not as explicit as the above examples, but it is quite general. We are grateful to Jesse Peterson for stimulating our interest in this question in the setting of lattices in Lie groups.

\begin{letterthm}
Consider a lattice $\Gamma$ in a connected semi-simple real Lie group $G$ with finite center. Then $\Gamma$ admits a singular subgroup which is virtually abelian.
\end{letterthm}

As we explain in Remark \ref{abelianexample}, if moreover $G$ has no compact factors and $\Gamma$ is torsion free and co-compact in $G$, then $\Gamma$ admits an abelian singular subgroup.

At this point, let us mention that all the former results on maximal amenability followed Popa's strategy of proving the maximal amenability of $Q \subset M$ by studying $Q$-central sequences in $M$. Namely the inclusion $Q \subset M$ was usually shown to satisfy the so-called ``asymptotic orthogonality property''. In contrast, our result relies on a new strategy, more specific to group von Neumann algebras, and completely different from Popa's approach.

The general idea in our approach is the following. Assume that $\Gamma$ acts on some compact space $X$. Then the maximal amenable subgroups of $\Gamma$ are stabilizers of probability measures on $X$. In non-commutative terms, one can more generally say that amenable subalgebras of $L\Gamma$ centralize states on the reduced C$^*$-algebraic crossed-product $C(X) \rtimes_r \Gamma$. The advantage of focusing our study on this crossed-product C$^*$-algebra is that it allows concrete computations. We will see at the end of this paper that this point of view also has a theoretical interest, providing new insight on solidity and strong solidity.

Finally, let us mention nice characterizations of singularity communicated to us by Narutaka Ozawa. We thank him for allowing us to include these characterizations here.

\begin{letterthm}[Ozawa]
\label{thmoz}
Consider an amenable subgroup $\Lambda$ of a discrete countable group $\Gamma$. The following are equivalent.
\begin{enumerate}
\item $\Lambda$ is a singular subgroup of $\Gamma$;
\item Every $\Lambda$-invariant state on $C^*_r(\Gamma)$ vanishes on $\Gamma \setminus \Lambda$;
\item For every $g \in \Gamma \setminus \Lambda$, we have that $0 \in \overline{\conv}^{\Vert \cdot \Vert}(\{\lambda(tgt^{-1}) \, , \, t \in \Lambda\}) \subset B(\ell^2\Gamma)$;
\item For any net $(\xi_n)$ of almost $\Lambda$-invariant unit vectors in $\ell^2\Gamma$ and all $g \in \Gamma \setminus \Lambda$, the inner product $\langle \lambda_g\xi_n,\xi_n\rangle$ goes to $0$.
\end{enumerate}
\end{letterthm}

Note that the last characterization is in the spirit of Popa's Asymptotic Orthogonality Property.

\subsection*{Acknowledgements} We first warmly thank Cyril Houdayer for important comments on an earlier version of this paper. We are grateful to Narutaka Ozawa for showing Theorem \ref{thmoz} to us and to Stefaan Vaes for providing us with a shortcut in the proof of Proposition \ref{centralizer}. We also thank Jean François Quint for an helpful discussion about lattices in algebraic groups and Adrian Ioana, Jesse Peterson and Sorin Popa for various discussions and remarks. We also thank the referees for their useful comments on the presentation and on Proposition \ref{centralizer}.

\section*{Notations and conventions}

All von Neumann algebras considered in the paper will be assumed to be separable. When a finite von Neumann algebra $M$ is being considered, we will denote by $\tau$ a faithful normal trace on it. In this case, $\Vert x \Vert_2 = \tau(x^*x)^{1/2}$ stands for the $2$-norm of an element $x \in M$, while $\Vert x \Vert_\infty$ denotes its operator norm. Also we denote by $L^2(M,\tau)$ (or sometimes just $L^2(M)$) the Hilbert space obtained by completion of $M$ with respect to the $2$-norm. When we want to view an element $x$ of $M$ inside $L^2(M)$, we will write $\hat x$. The left multiplication of $M$ on itself extends to a normal representation of $M$ on $L^2(M)$. Finally, denote by $J:L^2(M) \to L^2(M)$, the canonical anti-unitary that extends the map $x \in M \mapsto x^* \in M$.

If $M$ is a finite von Neumann algebra, and $N$ is a von Neumann subalgebra of $M$, we denote by $E_N$ the trace preserving conditional expectation.

If $\Gamma$ is a discrete countable group, we denote by $L\Gamma$ the associated von Neumann algebra; that is, the von Neumann subalgebra in $B(\ell^2\Gamma)$ generated by the left regular representation of $\Gamma$. Its commutant, $R\Gamma$, is also the von Neumann algebra generated by the right regular representation. Both $L\Gamma$ and $R\Gamma$ are finite and we choose $\tau$ to be the canonical trace on them, defined by $\tau(x) = \langle x\delta_e,\delta_e\rangle$ for all $x \in L\Gamma \subset B(\ell^2\Gamma)$ (here $\delta_e \in \ell^2\Gamma$ denotes the Dirac function at the identity element $e$ of the group $\Gamma$). In this case, we have that $L^2(L\Gamma,\tau) \simeq \ell^2(\Gamma)$ and $R\Gamma = J(L\Gamma)J$.

\section{Singular subgroups}

In this section, we prove Theorem \ref{criterion}. Let us first recall the definition of a singular subgroup.

\begin{defn}
Consider an amenable subgroup $\Lambda$ of a discrete countable group $\Gamma$. Suppose that $\Gamma$ acts continuously on the compact space $X$. We say that $\Lambda$ is \textit{\defin} in $\Gamma$ (with respect to $X$) if for any $\mu \in \Prob_\Lambda(X)$ and $g \in\Gamma\setminus \Lambda$, we have $g \cdot \mu \perp \mu$.  
\end{defn}

To illustrate this definition, let us mention the following trivial observation.

\begin{lem}
\label{groupmaxamen}
Consider a subgroup $\Lambda$ of a discrete countable group $\Gamma$. Then $\Lambda$ is maximal amenable inside $\Gamma$ if and only if there exists a continuous action $\Gamma \curvearrowright X$ on a compact space $X$ such that for any $\mu \in \Prob_\Lambda(X)$ and $g \in\Gamma\setminus \Lambda$, we have $g \cdot \mu \neq \mu$.
\end{lem}
\begin{proof}
If such a space $X$ exists, then $\Lambda$ is clearly maximal amenable.

Conversely, assume that $\Lambda$ is maximal amenable. For any $g \in\Gamma\setminus \Lambda$, the group $\langle \Lambda,g \rangle$ is not amenable: there exists a compact $\Gamma$-space $X_g$ such that $\Prob_{\langle \Lambda,g \rangle}(X_g) = \emptyset$. Replacing $X_g$ by a minimal $\Gamma$-invariant subset if necessary, we can assume that $X_g$ is a quotient of the Stone-Czech compactification $\beta\Gamma$ of $\Gamma$. 

Then we see that $X = \beta\Gamma$ does the job: if $\mu \in \Prob_\Lambda(X)$ then its push forward on $X_g$ is $\Lambda$-invariant, so it is not $g$-invariant by definition of $X_g$. Therefore $\mu$ is not $g$-invariant.
\end{proof}

\begin{rem}
\label{remsc}
  Observe that if $\Lambda$ is singular in $\Gamma$ with respect to the compact space $X$, then it is singular with respect to any closed $\Gamma$-invariant subset of $X$. Hence if $\Lambda$ is singular in $\Gamma$, it is also singular with respect to a minimal compact space. Therefore, arguing as in the above proof, we can see that an amenable group $\Lambda$ is singular inside $\Gamma$ with respect to some compact $\Gamma$-space $X$ if and only if it is singular with respect to the Stone-Czech compactification $\beta\Gamma$ of $\Gamma$.
\end{rem}

\begin{thm}[Theorem \ref{criterion}]
\label{criterionbis}
Suppose that $\Gamma$ is a discrete countable group admitting an amenable, singular subgroup $\Lambda$. Then for any trace preserving action $\Gamma \curvearrowright (Q,\tau)$ on a finite amenable von Neumann algebra, $Q \rtimes \Lambda$ is maximal amenable inside $Q \rtimes \Gamma$.
\end{thm}
\begin{proof}
We will denote by $M := Q \rtimes \Gamma \subset B(L^2(Q) \otimes \ell^2\Gamma)$ and by $N := Q \rtimes \Lambda$. Consider an intermediate amenable von Neumann algebra $N \subset A \subset M$. We will show that $A = N$.

Suppose that $\Lambda$ is \defin in $\Gamma$.
Since $A$ is amenable, there exists an $A$-central state $\varphi : B(L^2(Q) \otimes \ell^2\Gamma) \to \C$ whose restriction to $M$ coincides with the trace.

Take a unitary $u \in \cU(A)$. Using the fact that $u$ centralizes the state $\varphi$, we will show that $u \in N$. 

Fix $\varepsilon >0$. We denote by $\{v_g\}_{g\in\Gamma}$ the canonical unitaries implementing the action. Then  
by density, one can find $u_0 \in M$ of the form $u_0 = \sum_{g \in F} a_gv_g$, with $F\subset \Gamma$ finite and non-zero elements $a_g \in Q$ for all $g \in F$, such that $\Vert u^* - u_0 \Vert_2 < \varepsilon$.

We observe that $1 \otimes \ell^\infty(\Gamma) \simeq C(\beta\Gamma)$. With this identification, any element $f\in C(\beta\Gamma)$ commutes with $Q$, and we have that $v_gfv_g^* = \sigma_g(f)$ for all $g \in \Gamma$, where we denote by $\sigma_g$ the action induced by the canonical action of $\Gamma$ on $\beta\Gamma$. Since $\varphi$ is $\Lambda$-central, its restriction to $C(\beta\Gamma)$ is given by a $\Lambda$-invariant regular Borel probability measure $\mu$ on $\beta\Gamma$ such that $\varphi(f) = \int_{\beta\Gamma} fd\mu$, for all $f \in C(\beta\Gamma)$. By hypothesis $\Lambda$ is singular in $\Gamma$ and by Remark \ref{remsc}, it is singular with respect to $\beta\Gamma$. Hence for all $g \in F \setminus \Lambda$, the measures $\mu$ and $(g^{-1}\cdot \mu)$ are singular with respect to each other.

So there exist a compact set $K\subset\beta\Gamma$ and an open set $V\subset\beta\Gamma$ containing $K$ such that:
\begin{itemize}
\item $\mu(K) > 1 - \varepsilon$;
\item $\mu(gV) < (\varepsilon/(|F|\Vert a_g \Vert_\infty))^2$, for all $g \in F \setminus \Lambda$.
\end{itemize}
Urysohn's Lemma gives us a continuous function $f \in C(\beta\Gamma)$ supported on $V$ such that $0 \leq f \leq 1$ and $f = 1$ on $K$. 

On the one hand, since $\varphi$ is $u$-central, one derives 
\begin{align*}
\vert \varphi(u_0fu) \vert & = \vert \varphi(uu_0f) \vert \\
& \geq \vert \varphi(f) \vert - \vert \varphi((uu_0 -1)f) \vert \\
& \geq \mu(K) - \Vert u_0 - u^* \Vert_2\\
& > 1 - 2\varepsilon.
\end{align*}
For the third line above we used Cauchy-Schwarz inequality and the fact that $\varphi_{\vert M} = \tau$.

On the other hand, one computes
\begin{align*}
\vert \varphi(u_0fu) \vert  & \leq  \left\vert \sum_{g \in F \cap \Lambda} \varphi(a_gv_gfu) \right\vert + \sum_{g \in F \setminus \Lambda} \vert \varphi(a_g\sigma_g(f)v_gu) \vert \\
& \leq \vert \varphi(E_N(u_0)fu) \vert +  \sum_{g \in F \setminus \Lambda} \varphi(a_g\sigma_g(ff^*)a_g^*)^{1/2} \\
& \leq \Vert E_N(u_0) \Vert_2 +  \sum_{g \in F \setminus \Lambda} \Vert a_g \Vert_\infty \mu(gV)^{1/2}\\
& < (\Vert E_N(u) \Vert_2 + \varepsilon) + \varepsilon.
\end{align*}
Altogether, we get
\[ 1 - 2\varepsilon < \vert \varphi(u_0fu) \vert < \Vert E_N(u) \Vert_2 + 2 \varepsilon,\]
so that $\Vert E_N(u) \Vert_2 \geq 1 - 4 \varepsilon$. Since $\varepsilon$ was arbitrary, this implies that $u \in N$.
\end{proof}

\begin{rem} The above proof actually shows that any unitary $u \in L\Gamma$ which centralizes the C$^*$-algebra generated by $Q \rtimes_r \Gamma$ and $1 \otimes \ell^\infty(\Gamma)$ has to be contained in $N$. Note that this C$^*$-algebra is isomorphic to the crossed-product $(Q \otimes_r C(\beta\Gamma)) \rtimes_r \Gamma$ (where $\Gamma$ acts diagonally on $Q \otimes_r C(\beta\Gamma)$). We will present several applications of this point of view in Section \ref{sectionstabilizer}.
\end{rem}

Before actually providing examples of singular subgroups, let us prove Theorem \ref{thmoz}. 

\begin{thm}[Ozawa]
Consider an amenable subgroup $\Lambda$ of a discrete countable group $\Gamma$. The following are equivalent.
\begin{enumerate}
\item $\Lambda$ is a singular subgroup of $\Gamma$;
\item Every $\Lambda$-invariant state on $C^*_r(\Gamma)$ vanishes on $\lambda(\Gamma \setminus \Lambda)$;
\item For every $g \in \Gamma \setminus \Lambda$, we have that $0 \in \overline{\conv}^{\Vert \cdot \Vert}(\{\lambda(tgt^{-1}) \, , \, t \in \Lambda\}) \subset B(\ell^2\Gamma)$;
\item For any net $(\xi_n)$ of almost $\Lambda$-invariant unit vectors in $\ell^2\Gamma$ and all $g \in \Gamma \setminus \Lambda$, the inner product $\langle \lambda_g\xi_n,\xi_n\rangle$ goes to $0$.
\end{enumerate}
\end{thm}
\begin{proof}
(1) $\Rightarrow$ (2). Assume that $\Lambda$ is singular inside $\Gamma$ and take a $\Lambda$-invariant state $\varphi$ on $C^*_r(\Gamma)$. Fix $g \in G\setminus \Lambda$ and $\varepsilon > 0$. Since $\Lambda$ is amenable there exists a $\Lambda$-invariant state (also denoted by $\varphi$) on $B(\ell^2\Gamma)$ that extends $\varphi$. The restriction of $\varphi$ to $\ell^\infty(\Gamma)$ is a $\Lambda$-invariant state. By singularity, proceeding as in the proof of Theorem \ref{criterion}, we find $f \in \ell^\infty(\Gamma)$ such that $0 \leq f \leq 1$ and $\varphi(1 - f) \leq \varepsilon^2$ while $\varphi(g\cdot f) \leq \varepsilon^2$. Using Cauchy-Schwarz inequality, we get
\[\vert \varphi(\lambda_g) \vert \leq \vert \varphi(\lambda_g f)\vert + \varepsilon \leq \varphi(\lambda_g f \lambda_g^*)^{1/2}\varphi(f)^{1/2} + \varepsilon \leq 2\varepsilon.\]
As $\varepsilon$ was arbitrary, we indeed see that $\varphi(\lambda_g) = 0$.

(2) $\Rightarrow$ (3). Assume that there exists a $g \in \Gamma \setminus \Lambda$ for which the conclusion of (3) does not hold. By Hahn–Banach, there is $\varphi \in C^*_r(\Gamma)^*$ such that $\inf_{t \in \Lambda} \Re (\varphi(\lambda_{tgt^{-1}})) > 0$. By taking an average over $\{\varphi \circ \Ad(\lambda_t) \, , \, t \in \Lambda\}$, we may assume that $\varphi$ is $\Lambda$-invariant. Jordan decomposition now gives a $\Lambda$-invariant state on $C^*_r(\Gamma)$ such that $\varphi(g) \neq 0$. This violates (3).

(3) $\Rightarrow$ (4). Take a net $(\xi_n)$ of almost $\Lambda$-invariant unit vectors in $\ell^2\Gamma$ and $g \in \Gamma \setminus \Lambda$. Fix $\varepsilon > 0$. Assuming (3), we find $x \in \conv(\{\lambda(tgt^{-1}) \, , \, t \in \Lambda\})$ such that $\Vert x \Vert \leq \varepsilon$. By almost $\Lambda$-invariance, we get
\[ \limsup_n \vert \langle \lambda_g \xi_n , \xi_n \rangle \vert =  \limsup_n \vert \langle x \xi_n , \xi_n \rangle \vert \leq \varepsilon.\]

(4) $\Rightarrow$ (1). Suppose that $\Lambda$ is not singular inside $\Gamma$. Then we get a $\Lambda$-invariant state $\varphi$ on $\ell^\infty(\Gamma)$ and $g \in \Gamma \setminus \Lambda$ such that $\varphi$ is not singular with respect to $g \cdot \varphi$. Equivalently, this means that $\Vert \varphi - g \cdot \varphi \Vert_1 < 2$. Approximating $\varphi$ with normal states and using Hahn-Banach Theorem, we find a net of positive, norm one elements $\eta_n \in \ell^1(\Gamma)$ which is asymptotically $\Lambda$-invariant and satisfies $\Vert \eta_n - g \cdot \eta_n \Vert_1 \to \Vert \varphi - g \cdot \varphi \Vert_1 < 2$.

Define now $\xi_n = \eta_n^{1/2} \in \ell^2(\Gamma)$. Then these unit vectors are asymptotically $\Lambda$-invariant and yet we have the following fact showing that (4) does not hold.
\[\langle \lambda_g \xi_n , \xi_n \rangle = \frac{1}{2}(2 - \Vert \xi_n - \lambda_g \xi_n\Vert_2^2) \geq \frac{1}{2}(2 - \Vert \eta_n - g \cdot \eta_n\Vert_1) \not\to 0.\qedhere\] 
\end{proof}

\section{Examples}

\subsection{Hyperbolic groups}

As a first application of our criterion, we give a new proof of the main result of \cite{BC14}. Note that one can also recover the results from \cite{BC14} about {\it relatively} hyperbolic groups.

\begin{prop}[\cite{BC14}]
\label{BC14}
If $\Lambda$ is any infinite maximal amenable subgroup in a hyperbolic group $\Gamma$, then $\Lambda$ is \defin in $\Gamma$ with respect to the action $\Gamma \curvearrowright X = \pG$ on the Gromov boundary. In particular $L\Lambda$ is maximal amenable inside $L\Gamma$.
\end{prop}

Proposition \ref{BC14} is a direct consequence of the following classical result describing maximal amenable subgroups of a hyperbolic group $\Gamma$ in terms of the action $\Gamma \curvearrowright \pG$.

\begin{lem}
\label{hypgps}
Let $\Lambda < \Gamma$ be an infinite maximal amenable subgroup. The following facts are true.
\begin{enumerate}[(i)]
\item There exist two points $a,b \in \pG$ such that $\Lambda$ is the stabilizer of the set $\{a,b\}$, that is $\Lambda = \Stab_\Gamma(\{a,b\}) := \{ g \in \Gamma \, \vert \, g\cdot \{a,b\} = \{a,b\} \}$.
\item Any $\Lambda$-invariant probability measure on $\pG$ is of the form $t\delta_a + (1-t)\delta_b$ for some $t\in [0,1]$. 
\item Any element $g \in \Gamma \setminus \Lambda$ is such that $g \cdot \{a,b\} \cap \{a,b\} = \emptyset$.
\end{enumerate}
\end{lem}
We have already used this lemma in \cite{BC14} without giving a proof. Let us give a proof for completeness.
\begin{proof}
By \cite[Th\'eor\`eme 8.37]{GdH90}, $\Lambda$ is virtually cyclic. Denote by $h \in \Lambda$ an element of infinite order. Then by \cite[Th\'eor\`eme 8.29]{GdH90}, $h$ is a hyperbolic element: $h$ acts on $\pG$ with a north-south dynamics. Denote by $a$ and $b$ the attractive and repulsive points of $h$.

(i) Take $s \in \Lambda$. Then $shs^{-1}$ is a hyperbolic element with fixed points $s\cdot a$ and $s \cdot b$. If $\{a,b\}\cap \{s\cdot a,s\cdot b\}=\emptyset$, then by the ping-pong lemma, the group generated by $h$ and $shs^{-1}$ contains a free group. Since $\Lambda$ is amenable, this is not possible and hence $h$ and $shs^{-1}$ have at least a fixed point in common, say $a$. By \cite[Th\'eor\`eme 8.30]{GdH90} they also both fix $b$. Hence $\{s \cdot a,s \cdot b\} = \{a,b\}$ and so $\Lambda \subset \Stab_G(\{a,b\})$. The action of $\Gamma$ on the boundary is amenable \cite{Ad94}, so the equality follows from maximal amenability.

(ii) This is a consequence of the north-south dynamics action of $h \in \Lambda$.

(iii) The key result is \cite[Th\'eor\`eme 8.30]{GdH90}, which implies that any element which fixes one of the points $a$ or $b$ is in $\Lambda$. Take $g \in \Gamma$ such that $g\cdot a = b$.
Assume first that there exists $s \in \Lambda$ which exchanges $a$ and $b$. Then $sg$ fixes $a$ and so $g \in \Lambda$.
If all elements in $\Lambda$ fix $a$ and $b$, then $gsg^{-1}$ fixes $b$ and $g^{-1}sg$ fixes $a$ (so they both belong to $\Lambda$), for all $s \in \Lambda$. In that case, $g$ normalizes $\Lambda$ so $g \in \Lambda$ by maximal amenability.
\end{proof}

\begin{rem}
Take a maximal amenable subgroup $\Lambda$ of a hyperbolic group $\Gamma$. The proof of theorem A also shows that for any infinite subgroup $\Lambda_0$ of $\Lambda$, the von Neumann algebra $L\Lambda_0 \subset L\Gamma$ is contained in a unique maximal amenable subalgebra of $L\Gamma$, namely $L\Lambda$.
Our techniques can actually be extended to show that $L\Lambda$ is the unique maximal amenable subalgebra of $L\Gamma$ containing any given diffuse von Neumann subalgebra $Q \subset L\Lambda$. This goes in the direction of a conjecture of J. Peterson and A. Thom that states that any diffuse amenable subalgebra of a free group factor is contained in a unique maximal amenable subalgebra (see the last paragraph of \cite{PT11}).
\end{rem}

\subsection{Amalgamated free products and HNN extensions}

Using Bass-Serre theory, our criterion also applies to amalgamated free products.

\begin{prop}\label{AFP}
Let $\Lambda_1$ and $\Lambda_2$ be discrete groups (not necessarily finitely generated) with a common subgroup $\Lambda_0$. Put $\Gamma := \Lambda_1 \ast_{\Lambda_0} \Lambda_2$.
If $\Lambda_1$ is amenable and the index $[\Lambda_1:\Lambda_0] = \infty$ then $\Lambda_1$ is \defin in $\Gamma$. In particular $L\Lambda_1$ is maximal amenable inside $L\Gamma$.
\end{prop}
\begin{proof} 
Let us first construct the compact $\Gamma$-space $X$ for which we will verify the singularity property of $\Lambda_1 < \Gamma$. 
Assume that $\Gamma$ is as in the statement of Proposition \ref{AFP} and consider the Bass-Serre tree $T$ of $\Gamma$. By definition the vertex set of $T$ equals to $V(T) := \Gamma/\Lambda_1 \sqcup \Gamma/\Lambda_2$ and its edge set equals to $E(T) := \Gamma/\Lambda_0$, where the edge $g\Lambda_0$ relates $g\Lambda_1$ to $g\Lambda_2$.
By assumption the vertex $\Lambda_1$ has infinitely many neighbours.
In particular this tree is not locally finite. However every tree is by definition a uniformly fine hyperbolic graph in the sense of \cite[Section 8]{Bo12}, so one can still consider its visual boundary $\partial T$ and define a compact topology on $X := V(T) \cup \partial T$ as follows.

For $x,y \in X$ denote by $[x,y]$ the unique geodesic path between $x$ and $y$. If $x \in X$ and $A \subset V(T)$ is finite set of vertices, define 
\[M(x,A) := \{y \in X \, \vert \, [x,y] \cap A = \emptyset \}.\]

Then the family of sets $M(x,A)$ with $x \in X$, $A \subset V(T)$ finite, forms an open basis of a compact (Hausdorff) topology on $X$. See \cite[Section 8]{Bo12} for a proof or \cite[Section 2]{Oz06} for a short presentation of these facts.

Note that the action $\Gamma \curvearrowright X$ is continuous for this topology.

To prove Proposition \ref{AFP}, it is enough to show that the only $\Lambda_1$-invariant probability measure on $X$ is the Dirac measure $\delta_{\Lambda_1}$.

To that aim, assume that $\mu$ is a $\Lambda_1$-invariant probability measure on $X$. Note that since $\Lambda_0$ has infinite index in $\Lambda_1$, the vertex $\Lambda_1$ has infinitely many neighbors $\{g\Lambda_2\}_{g\in R}$, where $R \subset \Lambda$ is a section for the onto map $\Lambda_1 \to \Lambda_1/\Lambda_0$. Since $T$ is a tree, the open sets $\{M(g\Lambda_2,\{\Lambda_1\})\}_{g \in R}$ are disjoint and moreover $\Lambda_1$ acts transitively on these open sets. Hence they must have measure $0$ and therefore the probability measure $\mu$ has to be supported on their complement, namely $\{\Lambda_1\}$. 
\end{proof}

With the same proof we also get the following result.

\begin{prop}
Assume that $\Gamma = \HNN(\Lambda,\Lambda_0,\theta)$ is an HNN extension, where $\Lambda_0 < \Lambda$ and $\theta: \Lambda_0 \hookrightarrow \Lambda$ is an injective morphism. 
If $\Lambda$ is amenable and $[\Lambda:\Lambda_0] = [\Lambda:\theta(\Lambda_0)] = \infty$ then $\Lambda$ is \defin in $\Gamma$. 
\end{prop}
\begin{proof}[Sketch of proof]
  Assume that $\Gamma=\langle \Lambda, t | t^{-1}gt=\theta(g)\text{ for }g\in\Lambda_0\rangle$. As before, we will show that $\Lambda$ is singular in $\Gamma$ with respect to the compactification of the Bass-Serre tree $T$ of $\Lambda$. Let us describe the Bass-Serre tree. By definition the vertex set of $T$ equals to $V(T) := \Gamma/\Lambda$ and the vertex set equals to $E(T) :=\Gamma/\Lambda_0$, where the edge $g\Lambda_0$ connects $g\Lambda$ to $gt\Lambda$. The vertex $\Lambda$ has for neighbors all the vertices in one of the two families:
\begin{itemize}
\item The collection $\{gt\Lambda\}_{g \in R_1}$, where $R_1 \subset \Lambda$ is a section for the onto map $\Lambda \to \Lambda/\Lambda_0$;
\item The collection $\{gt^{-1}\Lambda\}_{g \in R_2}$, where $R_2 \subset \Lambda$ is a section for the map $\Lambda \to \Lambda/\theta(\Lambda_0)$.
\end{itemize}
These two families are infinite by our index assumptions. Moreover, $\Lambda$ acts transitively on each of these two families. So one can proceed exactly as in the previous proposition to deduce that any $\Lambda$-invariant probability measure on the compactification has to be supported on the vertex $\Lambda$. 
\end{proof}

In the finite index setting the result is false in general and the condition of Theorem \ref{criterion} is never satisfied. For instance assume that $\Gamma = \BS(m,n) = \langle a,t \, \vert \, ta^nt^{-1} = a^m \rangle$ with $m,n \geq 2$, and that $\Lambda = \langle a \rangle$. Then the conjugacy action of $\Lambda$ on $\Gamma \setminus \Lambda$ admits a finite orbit. Namely, $tat^{-1}$ has an orbit with $m$ elements, and so the element $x := \sum_{k=0}^{m-1} a^ktat^{-1}a^{-k} \in \C\Gamma$  commutes with $L\Lambda$.
In this case, $L\Lambda$ is not even maximal abelian in $L\Gamma$. However, one can check using Lemma \ref{groupmaxamen} that $\Lambda$ is maximal amenable inside $\Gamma$ as soon as $\vert m \vert ,\vert n \vert \geq 3$ (but it is not true for $\vert n \vert = 2$ or $\vert m \vert =2$).

\subsection{Lattices in semi-simple groups}

Finally, our criterion also allows to produce examples of a different kind, out of the (relatively)-hyperbolic world. 

\begin{prop}\label{slnz}
  For $n \geq 2$, put $\Gamma := \SL_n(\Z)$ and denote by $\Lambda$ the subgroup of upper triangular matrices in $\Gamma$. Then $\Lambda$ is singular in $\Gamma$. Moreover, $L\Lambda$ has a diffuse center.
\end{prop}
\begin{proof}
Put $G = \SL_n(\R)$ and denote by $P < G$ be the subgroup of upper triangular matrices, so that $\Lambda = \Gamma \cap P$. We will show that $\Lambda < \Gamma$ is singular with respect to the action on the homogeneous space $B = G/P$.
It is enough to prove that the unique $\Lambda$-invariant probability measure on $B$ is the Dirac mass on $[P]$. Fix $\mu \in \Prob_\Lambda(B)$.

Denote by $N < P$ the subgroup of unipotent matrices and put $\Lambda_0:=\Gamma\cap N$. Then \cite[Proposition 2.6]{Mo79} implies that the support of $\mu$ is pointwise fixed by the Zariski closure of $\Lambda_0$, namely $N$.
So we are left to check that $N$ has only one fixed point on $B$. Note that a point $g[P] \in B$ is fixed by $N$ if and only if  $g^{-1}Ng \subset P$. So let us take $g \in G$ such that $g^{-1}Ng \subset P$ and show that $g \in P$.

For a matrix $h\in G$, we denote with $\sigma(h)$ the spectrum of $h$. Observe that given $p \in P$ we have 
\[p \in N \text{ if and only if }\sigma(p)=\{1\}.\]
Since the spectrum is conjugacy invariant and $g^{-1}ng \in P$ for all $n \in N$, we have that $g^{-1}ng\in N$. Hence $g^{-1}Ng \subset N$, and this is even an equality because the nilpotent groups $N$ and $g^{-1}Ng$ have the same dimension. But a simple induction shows that the normalizer of $N$ in $G$ is $P$, so $g \in P$, as wanted.

For the moreover part, denote by $I$ the identity matrix and by $E_{1,n}$ the matrix with $0$ entries except for the entry row $1$/column $n$ which is equal to $1$. A simple calculation shows that the $\Lambda$-conjugacy class of $I + E_{1,n}$ is contained in $\{I \pm E_{1,n}\}$.
Therefore the center of $L\Lambda$ contains the element  $u + u^*$, where $u$ is the unitary in $L\Lambda$ corresponding to the element $I + E_{1,n} \in \Lambda$. Note that $I + E_{1,n}$ has infinite order, so $u$ generates a copy of $L\Z$. Finally $u +u^*$ generates a subalgebra of index $2$, which implies that $L\Lambda$ has diffuse center.
\end{proof}

Of course, the example given in the above proposition is not abelian (unless $n = 2$). We now turn to the question of existence of abelian, maximal amenable subalgebras in von Neumann algebras associated with lattices in semi-simple Lie groups.

\begin{prop}\label{casregulier}
Consider a lattice $\Gamma$ in a connected semi-simple real algebraic Lie group $G$ with finite center. Then there exists a virtually abelian subgroup $\Lambda$ in $\Gamma$ which is singular in $\Gamma$.
 \end{prop}
\begin{proof}
Before starting the proof, let us fix some notation. Denote with $d$ the real rank of $G$ and let $G=KAN$ be an Iwasawa decomposition of $G$, so that $K$ is a maximal compact subgroup, $A\cong \R^d$ and $N$ is nilpotent. Denote with $M$ the centralizer of $A$ in $K$. By \cite[Theorem 2.8]{PR72}, replacing $\Gamma$ by one of its conjugates if necessary, there exists an abelian subgroup $H \subset MA$ (a so-called {\it Cartan subgroup}) such that $H\cap \Gamma$ is cocompact in $H$. Moreover $H$ contains $A$, so it is co-compact in $MA$. Therefore $\Lambda_0:= MA\cap \Gamma$ is a co-compact lattice in $MA$ and it contains the abelian subgroup $\Gamma \cap H$ as a finite index subgroup. 

Let $P=MAN$ be a minimal parabolic subgroup. We will show that the normalizer $\Lambda := N_\Gamma(\Lambda_0)$ is singular in $\Gamma$ with respect to the action $\Gamma \curvearrowright G/P$. Consider a measure $\mu \in \Prob_\Lambda(G/P)$. 

{\bf Claim 1.} $\mu$ is supported on the set $F := \{x \in G/P \, \vert \, ax = x, \, \forall a \in A\}$.

The measure $\mu$ is $\Lambda_0$-invariant. Put $\tilde{\mu} := \int_M (g \cdot \mu) dg$, where $dg$ denotes the Haar probability measure on $M$.
Given an element $ma \in \Lambda_0$, with $m \in M$, $a \in A$, we see that 
\[a \cdot \tilde{\mu} =  \int_M (ag \cdot \mu) dg = \int_M (ga \cdot \mu) dg = \int (gma \cdot \mu)dg = \tilde{\mu}\]
Hence $\tilde{\mu}$ is invariant under the projection of $\Lambda_0$ on $A$. But this projection is a lattice in $A$. Applying \cite[Proposition 2.6]{Mo79}, we deduce that $\tilde{\mu}$ is supported on $F$.
Since $M$ commutes with $A$, the set $F$ is globally  $M$-invariant: $g \cdot \mu(F) = \mu(F)$ for all $g \in M$. Hence $1 = \tilde{\mu}(F) = \mu(F)$, as claimed.

{\bf Claim 2.} For all $x \in F$, we have $\Stab_G(x) \cap \Gamma = \Lambda_0$.

To prove this claim take $x \in F$, written $x = gP \in G/P$. Note that $\Stab_G(x) = gPg^{-1}$, so that $A \subset gPg^{-1}$. By \cite[Theorem 4.15]{BT65}, it follows that $gPg^{-1}=MAgNg^{-1}$. In particular $MA$ (and $\Lambda_0$) fixes $x$.

Take now $\gamma\in \Gamma \cap \Stab_G(x) = \Gamma \cap gPg^{-1}$. Write $\gamma=man$ with $m\in M$, $a\in A$ and $n\in gNg^{-1}$. By \cite[Proposition 8.2.4]{Zi84} there exists a sequence $(b_k)_k$ in $MA$ such that $b_knb_k^{-1}$ converges to the identity element $1_G$. Since $\Lambda_0$ is a uniform lattice in $MA$, there exists a subsequence $(b_{k_j})_j$ and a sequence $(c_{j})_j \subset \Lambda_0$ such that $b_{k_j}c_{j}^{-1}$ converges in $MA$. It is easy to conclude that $c_jnc_j^{-1}$ converges to the identity. Now, $c_{j}\gamma c_{j}^{-1}$ lies in $\Gamma$ and we have
\[c_j\gamma c_j^{-1} = c_{j}(man)c_{j}^{-1} = (c_{j}mc_{j}^{-1})a(c_{j}nc_{j}^{-1}), \, \text{ for all } j.\] 
But $c_{j}mc_{j}^{-1}$ belongs to the compact set $M$, so taking a subsequence if necessary, we see that $c_{j}\gamma c_{j}^{-1}$ converges to an element in $MA$. By discreteness of $\Gamma$, this implies that $c_{j}\gamma c_{j}^{-1} \in MA$ for $j$ large enough. Therefore $\gamma \in MA \cap \Gamma = \Lambda_0$ which proves Claim 2.

To prove that $\Lambda$ is singular in $\Gamma$, consider an element $g \in \Gamma$ such that $g \cdot \mu$ is not singular with respect to $\mu$. Then $g \cdot F \cap F \neq \emptyset$, so there exist two points $x,y \in F$ such that $y = gx$. This implies that $g\Lambda_0 g^{-1}$ fixes $y$, while $g^{-1}\Lambda_0g$ fixes $x$. From Claim 2 we deduce that $g$ normalizes $\Lambda_0$, so that $\Lambda$ is indeed singular in $\Gamma$.

To complete the proof of the proposition, it remains to show that $\Lambda_0$ has finite index inside $\Lambda$, which will ensure that $\Lambda$ is virtually abelian.

Assume that $g \in \Gamma$ normalizes $\Lambda_0$. Since $A$ lies in the Zariski closure of $\Lambda_0$, we have $gAg^{-1} \subset MA$. But $MA$ has a unique maximal $\R$-split torus, $A$. So $gAg^{-1} = A$ and $g$ normalizes $A$. Moreover $MA$ coincides with the centralizer $Z_G(A)$ of $A$ in $G$. Now we have only to observe that the Weyl group $N_G(A)/Z_G(A)$ is finite, see \cite[Section VII.7, item 7.84]{Kn96} for instance.
\end{proof}

\begin{rem}
For $\SL_3(\Z)$, let $\Lambda_0$ be the group generated by the following two commuting matrices: \[\begin{pmatrix}0&0&1\\1&0&-16\\0&1&8\end{pmatrix}\text{ and }\begin{pmatrix}81&4&-4\\-36&17&68\\4&-4&-15\end{pmatrix}.\]  

Then $\Lambda_0$ has finite index in a singular subgroup of $\SL_3(\Z)$. We do not know whether $\Lambda_0$ itself is singular in $\SL_3(\Z)$.
\end{rem}

\begin{rem}
\label{abelianexample}
We remark here that whenever $\Gamma$ is a co-compact and torsion free lattice of a real algebraic group $G$ without compact factors, there exists a free abelian subgroup of $\Gamma$ that is singular in $\Gamma$. In fact the authors in \cite{RS10} proved that under these hypothesis, the Cartan subgroup constructed in \cite{PR03} $H\subset G$ is such that $\Lambda_0:=\Gamma\cap H$ is isomorphic to $\Z^{\rk_\R(G)}$ and $\Lambda_0$ is malnormal in $\Gamma$. So we can use this Cartan subgroup in the proof of the above proposition to get that $\Lambda_0$ has finite index in a singular subgroup of $\Gamma$ and since $\Lambda_0$ is malnormal in $\Gamma$, then we must have that $\Lambda_0$ is singular in $\Gamma$. 
\end{rem}

\section{Amenable subalgebras as stabilizers of measures on some compact space}
\label{sectionstabilizer}

As explained in the introduction, the key of the above results is to view maximal amenable subalgebras of a group von Neumann algebra $L\Gamma$ as centralizers of states on some reduced crossed-product C$^*$-algebra $C(X) \rtimes_r \Gamma$. In this section we further develop this point of view and explain its link with more theoretical questions. What follows is largely inspired from the work of N. Ozawa on solidity \cite{Oz04,Oz10}.

The following proposition is the main ingredient. Our initial argument for (iii) relied on \cite[Lemma 5]{Oz04} and used exactness of the group in a redundant way. We are grateful to S. Vaes for suggesting to us a cleaner approach and to one of the referees for emphasizing and correcting a gap in an earlier version of the paper.

\begin{prop}
\label{centralizer}
Assume that $\Gamma$ is a countable discrete group which acts continuously on a compact space $X$. Denote by $B := C(X) \rtimes_r \Gamma$ the reduced crossed-product C$^*$-algebra. Consider a state $\varphi$ on $B$ which coincides on $C^*_r(\Gamma)$ with the canonical trace $\tau$.
The following are true.
\begin{enumerate}[(i)]
\item Given $x \in L\Gamma$ and $T \in B$, for every bounded sequence $(x_n)_n$ in $C^*_r(\Gamma) \subset B$ which converges strongly to $x$, the sequence $(\varphi(x_nT))_n$ converges and the limit depends only on $x$ and $T$. Therefore one can define $\varphi(xT) = \lim_n \varphi(x_nT)$ and similarly $\varphi(Tx)$.
\item The set $A_\varphi :=  \{x \in L\Gamma \, \vert \, \varphi(xT) = \varphi(Tx), \, \forall T \in B \}$ is a von Neumann subalgebra of $L\Gamma$.
\item If the action is topologically amenable in the sense of \cite{AD87} (see also \cite[Section 4.3]{BO08} for more on this), then $A_\varphi$ is amenable. Of course, every maximal amenable subalgebra of $L\Gamma$ arises this way.
\end{enumerate}
\end{prop}
\begin{proof} Before proceeding to the proof of (i)-(iii), let us fix some notations.
Denote by $(\pi_\varphi,H_\varphi,\xi_\varphi)$ the GNS triplet associated with $B$ and $\varphi$. Extend the state $\varphi$ to a normal state $\tilde \varphi$ on $\tilde B := \pi_\varphi(B)''$ by the formula $\tilde \varphi(x) = \langle x \xi_\varphi,\xi_\varphi \rangle$. Denote also by $\tilde C := \pi_\varphi(C^*_{r}(\Gamma))'' \subset \tilde B$. Note that $\tilde \varphi$ is a normal trace on $\tilde C$. Consider the central projection $p \in \tilde C$ that supports this trace, so that $\tilde \varphi$ is a faithful normal trace on $p\tilde C$. Then we see that the map $\sigma_0: C^*_r(\Gamma) \to p\tilde C$ defined by $\sigma_0(x) = p\pi_\varphi(x)$ for all $x \in C^*_{r}(\Gamma)$ is a trace-preserving $*$-morphism. Hence it extends to a normal $*$-isomorphism $\sigma: L\Gamma \to p\tilde C \subset p\tilde Bp$.

(i) Take $x \in L\Gamma$ and a sequence $x_n \in C$ that converges strongly to $x$, then we have that $\lim_n \varphi(x_nT) = \tilde \varphi(\sigma(x)\pi_\varphi(T))$, which does not depend on the choice of the sequence $(x_n)_n$. Indeed, since $\tilde \varphi(p) = 1$, we see that for all $n$,
\[\varphi(x_nT) = \tilde \varphi(\pi_\varphi(x_nT)) = \tilde \varphi(p\pi_\varphi(x_n)\pi_\varphi(T)) = \tilde \varphi(\sigma(x_n)\pi_\varphi(T)).\]
So the desired convergence is a consequence of the normality of $\sigma$ and $\tilde \varphi$.

(ii) From the formula obtained in (i), we see that 
\[A_\varphi = \{ x \in L\Gamma \, \vert \, \tilde \varphi(\sigma(x)\pi_\varphi(T)) = \tilde \varphi(\pi_\varphi(T)\sigma(x)), \forall T \in B\}.\]
Thus $\sigma(A_\varphi)$ is the intersection of $\sigma(L\Gamma) = p\tilde C$ with the centralizer of $\tilde \varphi$ in $p\tilde B p$. In particular $\sigma(A_\varphi)$ is a von Neumann algebra and so is $A_\varphi$.

(iii) Assume that the action is amenable. Then $B$ is nuclear, and so $\tilde B = \pi_\varphi(B)''$ is injective. In particular $p\tilde Bp$ is injective as well. Moreover, there exists a $\tilde \varphi$-preserving conditional expectation $E: p\tilde Bp \to \sigma(A_\varphi)$, because $\sigma(A_\varphi)$ centralizes $\tilde \varphi$. Hence $A_\varphi$ is amenable. The existence of $E$ follows from a standard argument that we include in the lemma below, for completeness.
\end{proof}

\begin{lem}
Consider a von Neumann algebra $M$ with a state $\varphi$ on it (not necessarily faithful). Take a von Neumann subalgebra $Q \subset M$ that centralizes $\varphi$ and assume that $\varphi$ is faithful and normal on $Q$ (so that it is a faithful normal trace on $Q$).
Then there exists a $\varphi$-preserving conditional expectation $E:M \to Q$.
\end{lem}
\begin{proof}
Given $x \in M$, define a sesquilinear form $\cB_x$ on $Q \times Q$ by the formula $\cB_x(a,b) = \tilde \varphi(b^*xa)$.

Then the Cauchy-Schwarz inequality gives $\vert \cB_x(a,b) \vert \leq \Vert x \Vert \Vert a \Vert_2 \Vert b \Vert_2$.
In particular $\cB_x$ induces a sesquilinear form on $L^2(Q,\varphi) \times L^2(Q,\varphi)$ and there exists a unique operator $T_x \in B(L^2(Q,\varphi))$ such that
\[\cB_x(\xi,\eta) = \langle T_x(\xi),\eta \rangle, \, \forall \xi, \eta \in L^2(Q,\varphi) \qquad \text{and} \qquad \Vert T_x \Vert \leq \Vert x \Vert.\]
Now we check that $T_x \in Q$. Take $y \in Q$ and $a,b \in Q$. We have
\[\langle T_x(ay),b \rangle = \varphi(b^*xay) = \varphi(yb^*xa) = \langle T_x(a),by^* \rangle = \langle T_x(a)y,b \rangle.\]
Therefore $T_x$ commutes with the right action of $y$. Since $y \in Q$ is arbitrary, we deduce that $T_x \in Q$.
The desired conditional expectation is then defined by the formula $E(x) := T_x$, for all $x \in M$.
\end{proof}

Let us provide some applications of Proposition \ref{centralizer} to solidity and strong solidity for bi-exact groups \cite{Oz04,OP10a,OP10b,CS13}.

\begin{defn}[\cite{BO08}, Section 15.2] 
\label{biexact}
A discrete group $\Gamma$ is {\it bi-exact} if there exists a compactification $X$ of $\Gamma$ such that
\begin{enumerate}
\item the left translation action of $\Gamma$ on itself extends to a continuous action $\Gamma \curvearrowright X$ which is topologically amenable;
\item the right translation action of $\Gamma$ on itself extends continuously to an action on $X$ which is trivial on the boundary $X \setminus \Gamma$.
\end{enumerate}
\end{defn}

For instance any hyperbolic group is bi-exact (because the Gromov compactification $\DG$ satisfies the above conditions).

Given a bi-exact group $\Gamma$, choose a compactification $X$ as in Definition \ref{biexact}. Since it is a compactification, we have inclusions $c_0(\Gamma) \subset C(X) \subset \ell^\infty(\Gamma)$. 
Denote by $\lambda$ and $\rho$ respectively the left and right regular representations of $\Gamma$ on $\ell^2\Gamma$, and define
\[B_\Gamma := C^*(C(X) \cup \lambda(\Gamma)) \subset B(\ell^2\Gamma).\]

By \cite[Proposition 5.1.3]{BO08}, $B_\Gamma$ is isomorphic to the reduced crossed product $C(X) \rtimes_{r} \Gamma$ by the left action of $\Gamma$. Moreover condition \ref{biexact}.(2) implies that $B_\Gamma$ commutes with $C^*_\rho(\Gamma)$ modulo compact operators:
\begin{equation}\label{compactcommutators}
[B_\Gamma,C^*_\rho(\Gamma)] \subset C^*(\lambda(\Gamma) \cdot [C(X),\rho(\Gamma)]) \subset C^*(\lambda(\Gamma) \cdot c_0(\Gamma)) \subset K(\ell^2(\Gamma)).
\end{equation}

We now show how solidity and strong solidity results can be deduced from Proposition \ref{centralizer}.

\begin{thm}[\cite{Oz04}]
If $\Gamma$ is bi-exact, then $L\Gamma$ is solid, meaning that the relative commutant of any diffuse subalgebra of $L\Gamma$ is amenable.
\end{thm}
\begin{proof}
Consider a sequence of unitaries $(u_n) \subset \cU(L\Gamma)$ which tends weakly to $0$. We will show that the von Neumann algebra $A$ of elements $x \in L\Gamma$ satisfying $\Vert [x,u_n] \Vert_2 \to 0$ is amenable.

Consider the state $\varphi$ on $B(\ell^2\Gamma)$ defined by \[\varphi(T) := \lim_{n \to \omega} \langle T\hat u_n, \hat u_n \rangle,\]
where $\omega$ is a free ultrafilter on $\N$. Note that $\varphi_{|L\Gamma} = \tau = \varphi_{|R\Gamma}$ and that $\varphi$ vanishes on the compact operators because $u_n$ tends weakly to $0$.
Applying Proposition \ref{centralizer}, we get that $A_\varphi =  \{x \in L\Gamma \, \vert \, \varphi(xT) = \varphi(Tx), \, \forall T \in B_\Gamma \}$ is an amenable von Neumann algebra. Let us show that $A \subset A_\varphi$.

Take $u \in \cU(A)$. By definition of $A$, we have for any $T \in B(H)$
\[\varphi(Tu) = \lim_{n \to \omega}  \langle T(uu_n),u_n\rangle = \lim_{n \to \omega}  \langle T(u_nu),u_n \rangle = \varphi(TJu^*J).\]
Similarly, we have $\varphi(uT)= \varphi(Ju^*JT)$.

Fix a bounded sequence $(x_k) \subset C^*_\rho(\Gamma)$ which converges strongly to $Ju^*J$. Since $\varphi_{|R\Gamma}$ is normal, the Cauchy-Schwarz inequality implies that
\begin{itemize}
\item $\lim_k \varphi(Tx_k) = \varphi(TJu^*J)$ and
\item $\lim_k \varphi(x_kT) = \varphi(Ju^*JT)$.
\end{itemize}
Now for each $k$, the operator $[T,x_k]$ is compact thanks to \eqref{compactcommutators}. Since $\varphi$ vanishes on compact operators we get 
\[\varphi(uT) = \varphi(Ju^*JT) = \lim_k \varphi(x_kT) = \lim_k \varphi(Tx_k) = \varphi(TJu^*J) = \varphi(Tu).\qedhere\]
\end{proof}

\begin{thm}[\cite{OP10a,CS13}]
\label{ssolidity}
If $\Gamma$ is bi-exact and weakly amenable (this is the case if $\Gamma$ is hyperbolic, \cite{Oz08}) then $L\Gamma$ is strongly solid, in the sense that the normalizer of a diffuse amenable subalgebra of $L\Gamma$ is amenable.
\end{thm}

Our proof still relies on the following weak compactness property due to Ozawa and Popa. The formulation is a combination of \cite[Theorem 3.5]{OP10a} and \cite[Theorem B]{Oz12} with the characterization of weak compactness given in \cite[Proposition 3.2(4)]{OP10a}.

\begin{thm}[\cite{OP10a,Oz12}]
\label{Oz12.B}
Assume that $\Gamma$ is weakly amenable. Then for any amenable subalgebra $A$ of $L\Gamma$, there exists a state $\varphi$ on $B(\ell^2\Gamma)$ such that
\begin{enumerate}
\item $\varphi(xT) = \varphi(Tx)$ for every $T \in B(\ell^2\Gamma)$ and $x \in A$;
\item  $\varphi(uJuJ T) = \varphi(T uJuJ)$ for every $T \in B(\ell^2\Gamma)$ and $u \in \cN_{L\Gamma }(A)$;
\item $\varphi(x) = \tau(x) = \varphi(Jx^*J)$ for every $x \in L\Gamma$.
\end{enumerate}
\end{thm}

The new part of the proof is the conclusion of strong solidity from the existence of such a state. It becomes extremely simple.

\begin{proof}[Proof of Theorem \ref{ssolidity}]
Assume that $A$ is a diffuse amenable subalgebra and consider a state $\varphi$ on $B(\ell^2\Gamma)$ as in Theorem \ref{Oz12.B}. By Proposition \ref{centralizer}, it suffices to show that $\cN_{L\Gamma }(A) \subset A_\varphi = \{x \in L\Gamma \, \vert \, \varphi(xT) = \varphi(Tx), \, \forall T \in B_\Gamma \}$.

First note that \cite[Lemma 3.3]{OP10b} implies that $\varphi$ vanishes on compact operators because $A$ is diffuse.

Take $u \in \cN_{L\Gamma }(A)$ and $T \in B_\Gamma$. By definition of $\varphi$, we have $\varphi(uJuJTJu^*J) = \varphi(Tu)$.

Fix a bounded sequence $(x_k) \subset C^*_\rho(\Gamma)$ which converges strongly to $Ju^*J$. Since $\varphi_{|R\Gamma}$ is normal, the Cauchy-Schwarz inequality implies that
\begin{itemize}
\item $\lim_k \varphi(uJuJTx_k) = \varphi(uJuJTJu^*J)$ and
\item $\lim_k \varphi(uJuJx_kT) = \varphi(uT)$.
\end{itemize}
Now for each $k$, the operator $uJuJ[T,x_k]$ is compact thanks to \eqref{compactcommutators}. Since $\varphi$ vanishes on compact operators we get 
\[\varphi(uT) = \lim_k \varphi(uJuJx_kT) = \lim_k \varphi(uJuJTx_k) = \varphi(uJuJTJu^*J) = \varphi(Tu).\qedhere\]
\end{proof}

Let us mention that one could also do a relative version of this strategy to prose relative strong solidity results. In particular, one could recover some of the results in \cite{PV14a,PV14b}. Note that the proof given in \cite{PV14b} also relies on bi-exactness explicitly.


\bibliographystyle{alpha1}

\end{document}